\documentclass[leqno,11pt]{amsart}

\usepackage{amsmath}
\usepackage{graphicx, color}
\usepackage{amscd}
\usepackage{amsfonts}
\usepackage{amssymb}
\usepackage{mathrsfs}
\usepackage{mathtools}
\usepackage{ulem}

\DeclareMathOperator*{\essinf}{ess\,inf}

\newtheorem{thm}{Theorem}[section]

\newtheorem{lem}[thm]{Lemma}
\newtheorem{prop}[thm]{Proposition}

\newtheorem{rem}{Remark}[section]

\numberwithin{equation}{section}
\newcommand{\ep}{\varepsilon}
\newcommand\be{\begin{equation}}
\newcommand\ee{\end{equation}}
\newcommand\R{\mathbb R}
\newcommand\Z{\mathbb Z}

\newcommand\ue{u_\delta}

\newcommand\fe{f_\delta}

\newcommand\rn{\R^n}
\newcommand\lcal{\mathcal{L}}
\newcommand\ltil{\tilde{\mathcal{L}}}
\newcommand\ld{\mathcal{L}_D}
\newcommand\lnd{\mathcal{L}_{ND}}
\newcommand\lu{\lcal[u]}
\newcommand\ldu{\ld[u]}
\newcommand\lndu{\lnd[u]}

\def\mp{{{\mathcal{M}}^+_{\lambda,\Lambda}}}
\def\mm{{{\mathcal{M}}^-_{\lambda,\Lambda}}}
\def\mpm{{{\mathcal{M}}^\pm_{\lambda,\Lambda}}}

\allowdisplaybreaks
\def\eps{\varepsilon}

\title[]
{The V\'azquez  maximum principle and the Landis conjecture for elliptic PDE with unbounded coefficients}

\author{Boyan Sirakov}
\author{Philippe Souplet}

\begin{document}

\begin{abstract}
We develop a new, unified approach to the following two classical questions on elliptic PDE:
\begin{itemize}
\item the  strong maximum principle for equations with non-Lipschitz nonlinearities,
\item the at most exponential decay of solutions in the whole space or exterior domains.
\end{itemize}
Our results apply to divergence and nondivergence operators with locally unbounded lower-order coefficients, in a number of situations where all previous results required bounded ingredients.
Our approach, which allows for relatively simple and short proofs, is based on a (weak) Harnack inequality with optimal dependence of the constants in the lower-order terms of the equation and the size of the domain, which we establish.
\end{abstract}

\maketitle


\section{Introduction}

\subsection{The setting} Let $\Omega\subseteq\rn$, $n\ge2$, be
an arbitrary domain in which is given a real-valued uniformly elliptic second order operator, either in divergence form
\begin{equation}\label{defdiv}
\ldu:=\mathrm{div}(A(x)Du +  b_1(x)u) + b_2(x){\cdot} Du +c(x) u,
\end{equation}
or in non-divergence form
\begin{equation}\label{defnondiv}
\lndu:=\mathrm{tr}(A(x)D^2u) +  b_1(x)\cdot Du +c(x) u,
\end{equation}
or more generally a fully nonlinear Hamilton-Jacobi-Bellman operator (i.e. a supremum or an infimum of operators as in \eqref{defnondiv}), for instance, an extremal operator of Pucci type
\begin{equation}\label{deffullynonl}
F[u]:=\mpm(D^2u)\pm b(x)|Du| + c(x) u.
\end{equation}
 Let $\lu$ denote any of \eqref{defdiv}--\eqref{deffullynonl}. We always assume that  $A(x)\in L^\infty(\Omega)$ satisfies
\begin{equation}\label{hypA}
\hbox{ there exist $0<\lambda\le \Lambda$ such that }
\lambda I\le A(x)\le \Lambda I,\ \ x\in \Omega; \quad A\in C(\Omega) \hbox{ if } \mathcal{L}=\lnd.
\end{equation}
The lower-order coefficients belong locally to Lebesgue spaces which make possible for weak solutions to satisfy the maximum principle and the Harnack inequality; specifically,
\begin{equation}\label{hypbc}
\hbox{$b,b_1,b_2\in L^q_{\mathrm{loc}}(\Omega)$ for some $q>n$, \ \ $c\in L^{p}_{\mathrm{loc}}(\Omega)$ for some $p>p_0$, where}
\end{equation}
\begin{equation}\label{defp0}
p_0=
\begin{cases}
n/2,&\quad\hbox{if $\mathcal{L}=\ld$}\\
p_E,&\quad\hbox{if $\mathcal{L}=\lnd$ or $F$}\\
\end{cases}
\end{equation}
and $p_E=p_E(n,\lambda, \Lambda)\in (n/2,n)$ is the constant from \cite{E}, \cite[Theorem 1.9]{Ca}.
In the following $\|b\|$ may denote $\|b_1\|$, or $\|b_1\|+\|b_2\|$, depending on the operator we consider. Also, by "(sub-, super-) solution" we mean that: (i) $u\in H^1_{\mathrm{loc}}(\Omega)$ in the case of $\ld$, and $u$ satisfies the (in)equality in the usual Sobolev sense (see  \cite[Chapter 8]{GT}); (ii) $u\in C(\Omega)$ in the case of $\lnd$ or $F$, and $u$ satisfies the (in)equality in the $L^q$-viscosity sense (see \cite{CCKS}).

 We study the following two classical questions.
\begin{itemize}
\item (V\'azquez strong maximum principle) Can the strong maximum principle hold for nonnegative solutions
 of $\lu\le f(u)$ if $f$ is not Lipschitz ?
\item (Landis conjecture) Is it true that solutions of $\lu=0$ in $\rn$ or in an exterior domain cannot decay super-exponentially at infinity ?
\end{itemize}
In spite of the extensive research in recent years, many open problems subsist (details will be given below). For instance, answers are almost entirely unavailable for operators with unbounded coefficients.

 To our knowledge, no connection between these two problems has been observed before. Here we present a new approach which unifies their treatment, and has the following main advantages.
\begin{itemize}
\item It gives answers for operators with (even locally) unbounded lower-order coefficients, in a number of situations where all previous results required bounded ingredients.
\item It extends many of the already available results on the Landis conjecture, even for equations with bounded coefficients; in particular, it proves the Landis conjecture for coercive fully nonlinear equations, a question which was completely open.
\item It treats simultaneously equations in divergence and non-divergence form, and provides rather short proofs.
\end{itemize}
The main tools of our method are the weak and the full Harnack inequalities, with optimal dependence of their constants in the lower-order terms
and the size of the domain, which we establish in Section \ref{secharn}.

\subsection{A V\'azquez type strong maximum principle}

A well-known result by V\'azquez
\cite{V} asserts that any nonnegative classical supersolution of
\be\label{vazeq}
\Delta u \le f(u)
\ee
in a domain $\Omega$ satisfies {\it the strong maximum principle} (SMP), i.e. either $u\equiv0$ or $u>0$ in $\Omega$, provided  $f(0)=0$, $f\ge 0$ is nondecreasing on $(0,\infty)$, and  $f$ satisfies the (sharp) condition:
\be\label{condint}
 \int_0 {(F(s))}^{-1/2}\, ds =\infty, \qquad F(s) =\int_0^s f(t)\, dt.
\ee
If $f$ has at most linear growth at zero, this is the standard SMP, but \eqref{condint} allows for non-Lipschitz nonlinearities, the most important and representative example being
\be\label{reprf}
f(s) = s\,|\!\ln s|^a, \quad a\le 2.
\ee

There has been a huge amount of work on extending the V\'azquez maximum principle to more general operators in \eqref{vazeq} and weak solutions, with extensions to quasilinear, fully nonlinear, singular, degenerate elliptic operators, see for instance \cite{PSZ}, \cite{PS}, \cite{FMQ}, \cite{FQS}, \cite{PR}, and the references therein. We refer to the book \cite{PS} for a very thorough presentation of this important extension of the SMP. Among many other things, Pucci and Serrin showed that it is sufficient that $f$ be nondecreasing only in a right neighborhood of zero (this type of extension is sometimes referred to in the literature as the Pucci-Serrin maximum principle). They showed also that condition \eqref{condint} is necessary for the strong maximum principle.

Almost all proofs of  V\'azquez type SMPs use the classical procedure of first proving a Hopf lemma by solving a radial problem. More specifically, thanks to the boundedness of the coefficients one can write an ODE whose solutions provide subsolutions of the given PDE in an annulus, with non-vanishing normal derivatives on the boundary. Then simple comparison provides the results -- this strategy has been applied  in all above quoted works. An exception is the paper~\cite{J} on {the  pure second order} equation div$(A(x)Du)= f(u)$, where measure-theoretic approach to the Harnack inequality is employed, in the style of \cite{BT}, to get an ODE on the volumes of
super-level sets of the solutions. A priori bounds for supersolutions of general equations with bounded coefficients can be found in \cite{Ko}.

A situation in which the radial/ODE approach does not seem to work is when the equation $\lu\le f(u)$ has {\it unbounded coefficients}, and this case has been completely open up to now. The following theorem settles it for nonlinearities as in \eqref{reprf}.

\begin{thm}\label{vazquez}
Assume \eqref{hypA}-\eqref{defp0}.
Let $u$ be a nonnegative weak supersolution of
\be\label{ineqnu}
\lu \le  f(u)\qquad\mbox{in }\; \Omega,
\ee
where $f$ is  continuous on $[0,\infty)$, $f(0)=0$, and
\be\label{condgrow}
\limsup_{s\to 0} \frac{f(s)}{s\,(\ln s)^2} <\infty.
\ee
If $\mathrm{ess\,inf}_{B}u=0$ for some ball $B\subset\subset\Omega$ then $u\equiv0$ in $\Omega$.
\end{thm}

To  our knowledge, this is the first result on SMP for  equations with unbounded coefficients and non-{L}ipschitz nonlinearities. In addition, we do not have any condition of monotonicity of $f$ in a right neighborhood of zero.

 In the recent work \cite{NS} on solvability of general non-coercive fully nonlinear equations with quadratic dependence in the gradient, all results had to be restricted to bounded coefficients precisely because of the lack of a SMP of V\'azquez type for the nonlinearity $f(s) = s\,|\!\ln s|$ and lower-order coefficients in $L^q$, $q>n$, which is the natural integrability for the framework of \cite{NS} (see the end of \cite[Section 2]{NS}). So Theorem \ref{vazquez} extends all results from \cite{NS} to unbounded coefficients.

Let us sketch the main point of the proof of Theorem \ref{vazquez}, assuming for simplicity $f(s) = s\,|\!\ln s|^a$, $a\le 2$, and  $u$ continuous.
Since the zeroset of $u$ is closed, it  suffices to show that it is also open.
If $u$ vanishes at some point, say $0$, then, for each small $\delta,r>0$, the rescaled function $v=u(rx)+\delta$
satisfies $v(0)=\delta$ and $v$ is a
positive solution of a linear equation whose main zero order term is of the form $d(x)v$ with $d(x)=r^2f(u(rx))/(u(rx)+\delta)$.
By the (weak) Harnack inequality with sharp dependence on the zero order coeffcient,
we can estimate the integral of a suitable power of $v$ over the unit ball $B_1$ by the quantity
$$N_\delta=v(0)\exp\bigl[C\|d\|_\infty^{1/2}\bigr]=\delta\exp\Bigl\{Cr\bigl\||\!\ln(u)|^{a/2}\sqrt{u/(u+\delta)}\bigr\|_\infty\Bigr\},$$
where the sup norm is taken on $B_2$.
But it can be seen that $\bigl\||\!\ln(u)|^{a/2}\sqrt{u/(u+\delta)}\bigr\|_\infty=O\bigl[\ln^{a/2}(1/\delta)\bigr]$ as $\delta\to 0$.
When $a\le 2$, by choosing $r>0$ small enough, it follows that $N_\delta$ goes to $0$ as $\delta\to 0$;
hence $u$ vanishes in a neighborhood of $0$ and the zeroset of $u$ is open.

It is an open problem whether \eqref{condgrow} can be replaced by \eqref{condint} in Theorem \ref{vazquez}. We remark that \eqref{condint} is a typical ODE hypothesis, while it seems difficult to find an ODE argument in the presence of unbounded coefficients, as we explained above.

\subsection{The Landis conjecture} In \cite{KL}, among many other things, Kondratiev and Landis asked whether a solution of a uniformly elliptic PDE with bounded coefficients in an exterior domain must necessarily be trivial, 
provided that it decays as $|x|\to\infty$ more rapidly than $\exp(-C_0|x|)$ for a sufficiently large constant $C_0$. This property is known as "the Landis conjecture", also in its sharper form where the optimal $C_0$ is sought for, or in a weaker form  brought up by Kenig in \cite{K2}, where the decay to rule out is $\exp(-|x|^{1+\epsilon})$, $\epsilon>0$. Landis conjecture can also refer to entire solutions (i.e. defined in the whole space).

The Landis conjecture has a long history, in particular in the case $n=2$. Meshkov disproved it for complex potentials $c(x)$ and entire complex solutions of $\Delta u + c(x) u=0$, showing the optimal decay to be $\exp(-|x|^{4/3})$. Important quantitative extensions of Meshkov's result, as well as extensions to more general equations, can be found in \cite{BK}, \cite{EK}, \cite{K2}, \cite{DZ}, \cite{B1}. All these works use Carleman type estimates, which do not distinguish between real and complex solutions.

The Landis conjecture is still open for equations with bounded real coefficients, even for entire solutions of $\Delta u + c(x) u=0$; however a lot of work has been done in the last years. In \cite{KSW} Kenig, Silvestre and Wang prove the weak form of the Landis conjecture in $\mathbb{R}^2$, for $\ld$ with bounded coefficients and one of the $b_i=0$, under the hypothesis that $c(x)\le0$; actually they obtain a more precise quantitative bound, saying that within distance one of each point on the sphere $|x|=R$ there is a point at which $|u|$ is at least $\exp(-C_0R (\log R))$. They also prove a bound in $\exp(-C_0R (\log R)^2)$ for solutions in exterior domains of $\mathbb{R}^2$. This paper brought a number of generalizations, see \cite{DKW}, \cite{KW}, \cite{DZ},  and the references therein. All these works are for $n=2$ and equations in divergence form, and make various hypotheses
 on the lower-order coefficients of $\ld$ which in particular imply that $\ld$ or its dual satisfy the maximum principle on bounded subdomains.

Recently, Rossi \cite{R} established sharp versions of the Landis conjecture for general linear non-divergence form operators with bounded ingredients, either for radial coefficients, or for radial solutions, or  under the hypothesis that $\lnd$ satisfies the maximum principle on bounded subdomains and the solution has a sign on the boundary if the latter is not empty. The proof of the non-radial case in \cite{R} relies heavily on the fact that $e^{- M|x|}$ is a subsolution of the operator for sufficiently large $M>0$, 
a property which holds only if the coefficients of the operator are bounded. Variants of some of the results in \cite{R} are obtained among other things in the earlier paper \cite{ABG} via probability techniques, and in the recent  work \cite{LeB} via a duality argument (due to M. Pierre).

The only result on the Landis conjecture in the real case
that does not make some hypothesis on the coefficients which implies the validity of the maximum principle is the very recently posted paper \cite{LM} which settles Kenig's weak form of the Landis conjecture in dimension $2$, for entire solutions of $\Delta u + c(x) u=0$ and a bounded $c(x)$.

In the real coefficients case, unbounded $b_i$ (with $c$ bounded) are considered in \cite{KW}, \cite{DW}, for divergence form operators and $n=2$ only, under the restrictions that  $b_i$ are integrable at infinity, i.e.~belong to $L^q(\mathbb{R}^2)$, $q>2$, and that $|u|$ grows at most like $\exp(C_0|x|^\alpha)$ with $\alpha=1- 2/q\in (0,1)$.
These rather strong hypotheses lead to a different Landis type result with stronger conclusion,
ruling out solutions that decay like $\exp(-C_1|x|^{\alpha+})$.

Our goal here is to prove the Landis conjecture in $\rn$ in any dimension,
for unbounded lower-order coefficients which are only uniformly locally integrable  (and thus bounded coefficients are a very particular case), under the hypothesis that the maximum principle holds in any bounded subdomain. We also consider exterior domains.

Our method is completely different from the previous works, and allows for rather short proofs. It permits us
to treat simultaneously divergence and non-divergence equations; for the latter we do not know of any previous results with unbounded coefficients. We also consider fully nonlinear equations, for which no previous results are available at all. We use  only the sharp form of the weak and full Harnack inequalities together with the comparison principle and the solvability of the Dirichlet problem in bounded domains.

We recall the definition of uniformly local Lebesgue spaces.
If $h\in L^s_{\mathrm{loc}}(\overline{\Omega})$, $1\le s \le\infty$, we say that $h\in L^s_{ul}(\Omega)$ provided the quantity (norm)
\begin{equation}\label{deful}
\|h\|_{L^s_{ul}(\Omega)}:=\sup_{x\in \rn} \|h\|_{L^s(\Omega\cap B_1(x))}
\end{equation}
is finite. The spaces $L^s_{ul}$ have been used for instance in \cite{Ka, GV}.
Note that $L^{s_2}_{ul}(\Omega)\subset L^{s_1}_{ul}(\Omega)$ if $1\le s_1\le s_2\le \infty$, $\Omega\subseteq\rn$.
Also, we call exterior domain any $\Omega$ such that
$B_{r_1}\subset \rn\setminus \Omega\subset B_{r_2}$ for some $r_2>r_1>0$.
We  will not assume any smoothness on $\partial\Omega$.

\begin{thm}\label{landis} Let $\Omega=\rn$ or $\Omega$ be an exterior domain.
Assume \eqref{hypA}, $b, b_1, b_2 \in L^q_{ul}(\Omega)$, $c\in L^p_{ul}(\Omega)$, with $n<q\le\infty$, $p_0<p\le\infty$, and \eqref{defp0}.
Assume also that $\lcal$ satisfies the maximum principle in each bounded subdomain of $\Omega$.
Then there exists a constant $C_0=C_0(n,p,q,\Lambda/\lambda)$ such that if $u$ is a solution of
\begin{equation}\label{BClandis}
\hbox{$\lu=0$ in $\Omega$, \quad with $u\ge 0$ on $\partial \Omega$ or $u\le 0$ on $\partial \Omega$ (if $\partial \Omega$ is not empty),}
\end{equation}
and
\begin{equation}\label{concllandis}
\lim_{|x|\to \infty} e^{C_1|x|}|u(x)|
=0, \quad\hbox{ with } C_1:= C_0\left(1+\|b\|_{L^q_{ul}(\Omega)}^{\frac{1}{1-(n/q)}} + \|c\|_{L^{p}_{ul}(\Omega)}^{\frac{1}{2-(n/p)}}\right),
\end{equation}
then $u\equiv 0$.
\end{thm}

\begin{rem} \label{remTh2}
(i) We prove this theorem for Hamilton-Jacobi-Bellman operators which include \eqref{deffullynonl} and \eqref{defnondiv} as particular cases (see Section \ref{secprooflandis}).
Also, assumption \eqref{concllandis} can be weakened to
\begin{equation}\label{concllandis2}
\liminf_{R\to \infty}\ e^{C_1R} \sup_{|x|=R} |u(x)| =0.
\end{equation}

(ii)  By definition, $\ld$ satisfies the maximum principle in a domain $G$ if $\ldu\le 0$ in $G$ and $u^-\in H^1_0(G)$ imply $u^-=0$ in $G$. In the non-divergence case $F$ satisfies the maximum principle if $F[u]\le (\ge)\ 0$ in $G$ (resp. $\lnd[u]\le (\ge)\ 0$) and $u\ge(\le)\ 0$  on $\partial G$ implies $u\ge(\le)\ 0$ in $G$, for each $L^q$-viscosity (sub-/super-)solution $u\in C(\overline{G})$.

As for the boundary conditions in \eqref{BClandis} for the exterior domain case,
they are understood in the standard sense. Namely, $u\le 0$ on $\partial\Omega$ means
$(\varphi u)_+\in H^1_0(\Omega)$ for all $0\le\varphi\in C^\infty_0(\rn)$, if $\mathcal{L}=\ld$;
$u\in C(\overline\Omega)$ and $u(x)\le 0$ for all $x\in\partial\Omega$, if $\mathcal{L}=\lnd$ or $F$.

 We also recall that
 by interior De Giorgi-Moser estimates (see \cite[Theorem 8.24]{GT}),
any solution of $\ldu=0$ is (H\"older) continuous in $\Omega$.

\smallskip

(iii) It is classical that the maximum principle is satisfied in a bounded domain by $\lnd$ if $c\le 0$ and by $\ld$ if $c+\mathrm{div}(b_1)\le 0$  in the sense of distributions,
but this condition is of course far from necessary.
 Various more general results are available, see for instance \cite{T2} for divergence form equations, \cite{BNV}, \cite{BR} for linear equations with bounded coefficients, \cite[Prop. 3.4]{S1} for fully nonlinear equations with unbounded coefficients.
It is also well known
that the validity of the maximum principle can be related to the positivity of the
{\it first eigenvalue} of the operator, or to the existence of a strictly positive supersolution. See \cite{Ch}, \cite{BNV}, \cite{QS}, \cite{Ar}.
\end{rem}

Here is the main idea of the proof of Theorem \ref{landis}. Under our assumption that the operator $\mathcal{L}$
satisfies the maximum principle in bounded subdomains, one can first show the existence of a positive solution $\psi$ of $\mathcal{L}[\psi]=0$ in $\Omega$ (cf.~Propositions~\ref{prop1} and \ref{prop2}). Next, from the Harnack inequality with sharp dependence on the size of the domain
(and on the coeffcients), we deduce a precise lower exponential bound on the decay of $\psi$ at infinity. Then, for a general (possibly sign-changing) solution $u$ of $\mathcal{L}[u]=0$ in $\Omega$, if $|u|$ decays faster than $\psi$ at infinity, one may apply the comparison
principle to $\pm u$ and $\delta\psi$ on the intersection of $\Omega$ with a large ball, for each $\delta>0$.
Letting $\delta\to 0$, we conclude that $u$ has to vanish identically.

\medskip

The rest of the article is organized as follows. Section~2 is devoted to the statement and proof of the
sharp, weak and full, Harnack inequalities. Theorems~\ref{vazquez} and \ref{landis} are respectively proved in Sections~3 and 4.
In the Appendix, for the reader's convenience and in order to supply a full quotable source, we provide a proof
of the usual Harnack inequality, under general hypotheses.

\section{On the Harnack inequality}\label{secharn}

We start by recalling the following classical "half-Harnack" inequalities.

\medskip
\noindent{\bf Theorem A.}
{\it Let $\Omega=B_2$. Assume \eqref{hypA}, $b, b_1, b_2 \in L^q(B_2)$, $c, g\in L^p(B_2)$, with $q>n$, $p>p_0$, and \eqref{defp0}. Suppose $\|b\|_{L^q(B_2)}\le 1$, $\|c\|_{L^{p}(B_2)}\le 1$.
\begin{itemize}
\item (weak Harnack inequality)
There exist constants $\epsilon, C_0>0$ depending  only on $n,p,q,$ $\lambda,\Lambda$, such that if
$ u\ge 0$  
satisfies $\lu\le g$ in $B_2$, then
\begin{equation}\label{weakha}\left(\int_{B_{3/2}} u^\epsilon\,dx\right)^{1/\epsilon}\le C_0\left( \inf_{B_1} u + \|g\|_{L^{p}(B_2)}\right).
\end{equation}
\item (local maximum principle)  For each $\varepsilon>0$, there exists a constant $C_\varepsilon>0$ depending only on $n,p,q,\lambda,\Lambda,\varepsilon,$ such that, if $u$ satisfies $\lu\ge g$ in $B_2$, then
\begin{equation}\label{locma}\sup_{B_1} u\le   C_\varepsilon \left(\left(\int_{B_{3/2}} |u|^\varepsilon\,dx\right)^{1/\varepsilon} + \|g\|_{L^{p}(B_2)}\right).\end{equation}
\end{itemize}
}

In this generality, this theorem was proved in \cite{T2} for divergence form operators, and in \cite{KSh}, \cite{KSl} for fully nonlinear operators.

\begin{rem} If $\lcal=\ld$ is in divergence form we can add to the right-hand side of the differential inequality a term div$(h)$, for some $h\in L^q(B_2)$ (note $div(L^q)\subset H^{-1}$ if $q>n$), adding also $\|h\|_{L^{q}(B_2)}$ to the right-hand side of the inequalities in Theorem A. \end{rem}

\begin{rem} In \cite{KSh}, \cite{KSl} the results are actually stated for $c=0$, but extension to arbitrary $c\in L^{p}$, $p>p_0$ is rather straightforward.
Specifically, $F[u]\le g$ and $u\ge0$ imply $\mathcal{M}_{\lambda,\Lambda}^-(D^2u) -b|Du|-c^-u\le g$, and because of the sign $c^-\ge0$ for the latter operator the ABP inequality holds without difference with respect to the case $c=0$. On the other hand $F[u]\ge  g$ implies $\mathcal{M}_{\lambda,\Lambda}^+(D^2u) +b|Du|\ge -cu + g$ and we can treat $cu$ as a right-hand side, through a well-known argument.
Nevertheless, since Theorem A plays a pivotal role in our study, for the reader's convenience and in order to supply a full quotable source, we provide a proof in the appendix. \end{rem}
\smallskip

An essential tool in our analysis are the following  Harnack type inequalities with sharp dependence in the lower-order coefficients and the size of the domain.

 For any $r$ with $n<r\le \infty$, we set
$$
\beta_r= \frac{r}{r-n} = \frac{1}{1-(n/r)}, \qquad  \gamma_r= \frac{r}{2r-n} = \frac{1}{2-(n/r)}, \qquad \beta_\infty=1,\quad\gamma_\infty=\frac{1}{2},
$$
and we denote by $G_R\subset G_R^\prime$, $R>2$, either $G_R=B_R, G_R^\prime=B_{R+1}$ or  $G_R=B_R\setminus B_2, G_R^\prime=B_{R+1}\setminus B_1$.

\begin{thm}\label{sharpharn}
 Let $\Omega=G'_R$ for some $R>2$. Assume \eqref{hypA},  $b, b_1, b_2 \in L^q(G'_R)$, $c, g\in L^p(G'_R)$, with $n<q\le\infty$, $p_0<p\le\infty$, and \eqref{defp0}. Set
\begin{equation}\label{defAR}
A=A_R:= 1+ \|b\|^{\beta_q}_{L^q_{ul}(G_R^\prime)}+\|c\|^{\gamma_p}_{L^p_{ul}(G_R^\prime)}.
\end{equation}
There exist constants $\epsilon, C_0>0$ depending only on $n,p,q,\lambda,\Lambda$, such that  the following holds.
\begin{itemize}
\item (weak Harnack inequality) If $u\ge0$ satisfies $\lu\le g$ in $G_R^\prime$, then
\begin{equation}\label{weakHarnack}
\left(\int_{G_R} u^\epsilon\,dx\right)^{1/\epsilon}\le e^{C_0AR}\left( \inf_{G_R} u + \|g\|_{L^{p}_{ul}(G_R^\prime)}\right).
\end{equation}
\smallskip
\item (local maximum principle) If $u$ satisfies $\lu\ge g$ in $G_R^\prime$, then, for each $\varepsilon>0$,
\begin{equation}\label{weakMP}
\sup_{G_R} u\le  C_\varepsilon\left(A^{n/\varepsilon}\left(\int_{G_R^\prime} |u|^\varepsilon\,dx\right)^{1/\varepsilon} + \|g\|_{L^{p}_{ul}(G_R^\prime)}\right),
\end{equation}
for some constant $C_\varepsilon>0$ depending only on $n,p,q,\lambda,\Lambda,\varepsilon$.
\smallskip
\item (Harnack inequality) If $u\ge0$ satisfies $\lu= g$ in $G_R^\prime$, then
\begin{equation}\label{sharpHarnack}
\sup_{G_R} u\le e^{C_0AR}\left( \inf_{G_R} u + \|g\|_{L^{p}_{ul}(G_R^\prime)}\right).
\end{equation}
\end{itemize}
\end{thm}

\begin{rem} The optimality of the constant in \eqref{weakHarnack},  is obvious from the ODE $u^{\prime\prime} -2bu^{\prime}- cu=0$, $b,c\in \mathbb{R}^+$, with solution $u(x)=e^{Dx}$, $D=b+\sqrt{b^2+c}$.
\end{rem}

\begin{rem} In the particular case when the coefficients of the operator are bounded and the operator is in divergence form, the constant in \eqref{sharpHarnack} appears as a remark without proof after Theorem 8.20 in \cite{GT} (with $G_R=B_R, G_R^\prime=B_{4R}$, $\|g\|_{L^{p}(B_{4R})}$ instead of $\|g\|_{L^{p}_{ul}(G_R^\prime)}$). We note however that a straightforward examination of the constants in the proofs of Theorems 8.17-8.18 in \cite{GT} does not seem to give exactly that dependence in the norms of the coefficients, and a refinement is needed. More specifically, with $R=1$, following the constants in those theorems one gets a Harnack constant which grows like $\nu^{C(n)\sqrt{\nu}}$ as $\nu \to \infty$ and not $C(n)^{\sqrt{\nu}}$ as we have above ($\nu$ in \cite{GT} grows like $\|b\|_\infty+\sqrt{\|c\|_\infty}$ here, see (8.6) in \cite{GT}, their $d$ is our $c$). Note this is in accordance with what we find in other articles which use the same technique, for instance \cite[Remark 2.2]{IKR}, where the quoted constant is also like $\nu^\nu$. We also observe that there is a misprint in \cite[Problem 8.3]{GT}, the correct statement of that Problem  is with $\sqrt{\Lambda/\lambda} + \nu R$, and not $\sqrt{\Lambda/\lambda + \nu R}$ --  see the previous remark. Finally, in the non-divergence case, an examination of the constants in Safonov's proof of the weak Harnack inequality in \cite{Saf} yields an exponential in which the $L^\infty$-norms of the coefficients are taken to some (possibly large) power. \end{rem}

\begin{rem} We obtain the dependence in $R$ in Theorem \ref{sharpharn} not from a rescaling $x\to x/R$ but from a "Harnack chain" of balls of  fixed radius, which leads to a Harnack inequality in which $G_R^\prime\setminus G_R$ has in-radius of order $1$ instead of $R$, and whose constants depend on norms of the  coefficients in $L^q_{ul}(G_R^\prime)$ instead of $L^q(G_R^\prime)$ (note even for constant functions the latter norm degenerates as $R\to\infty$ while the former does not).
\end{rem}

\begin{rem} As in Theorem A, if $\lcal=\ld$ is in divergence form we can add to the right-hand side of the differential (in)equalities in Theorem \ref{sharpharn} a term div$(h)$, for some $h\in L^q(G_R^\prime)$, adding also $\|h\|_{L^{q}(G_R^\prime)}$ to the right-hand side of \eqref{weakHarnack}-\eqref{sharpHarnack}. Furthermore, the proof of the local maximum principle below shows that \eqref{weakMP} can be replaced by the
more precise estimate
\begin{equation}\label{weakMPul}
\sup_{G_R} u\le  C_0\Bigl(A^{n/\varepsilon}\|u\|_{L^{\varepsilon}_{ul}(G_R^\prime)} +
A^{(n/p)-2}\|g\|_{L^{p}_{ul}(G_R^\prime)}\Bigr)
\end{equation}
(where $\|\cdot\|_{L^{\varepsilon}_{ul}}$ is still defined by \eqref{deful}, although it  need not be a norm in case $\varepsilon<1$).
 \end{rem}

{\it Proof of Theorem \ref{sharpharn}}.
In all that follows $C_0>0$ depends on $n, p,q,\lambda,\Lambda$, and may change from line to line.
\smallskip

{\bf Step 1.} {\it Weak Harnack inequality in small balls.}
Let $x_0\in\R^n$, $r_0>0$ and set $B=B_{r_0}(x_0)$, $B'=B_{2r_0}(x_0)$.
Assume $b\in L^q(B')$, $c\in L^p(B')$ and
\begin{equation}\label{r0small}
 0<r_0\le \Bigl[2+\|b\|_{L^q(B')}^{\beta_q}+\|c\|_{L^p(B')}^{\gamma_p}\Bigr]^{-1}\in (0,1/2].
\end{equation}
If $u\ge0$ satisfies $\lu\le g$ in $B'$, then we have
$$\left(\int_{B} u^\epsilon\,dx\right)^{1/\epsilon} \le C_0 r_0^{n/\epsilon} \left( \inf_{B} u +  r_0^{2-n/p}\|g\|_{L^{p}(B')}\right).$$

{\it Proof of Step 1.} Let $v(y)=u(x_0+r_0y)$ for $y\in B_2$ (hence $x_0+r_0y\in B'$).
The function $v$ satisfies
$\ltil[v]\le \tilde{g}$
in $B_2$,
where the coefficients of the modified operator $\ltil$ are $\tilde A(y)= A(x_0+r_0 y)$,
$\tilde b(y)=r_0 
b(x_0+r_0 y)$, $\tilde c(x)=r_0^2c(x_0+r_0y)$, and $\tilde{g}(x)=r_0^2g(x_0+r_0y)$. We compute
$$\begin{aligned}
\|\tilde b\|_{L^q(B_2)}
&=r_0\Bigl(\int_{|y|<2} |b(x_0+r_0y)|^q\,dy\Bigr)^{1/q}\\
&= r_0^{1-n/q}\Bigl(\int_{|x-x_0|<2r_0} |b(x)|^q\,dx\Bigr)^{1/q}
 = r_0^{1-n/q}\|b\|_{L^q(B')}
\end{aligned}
$$
and similarly $\|\tilde c\|_{L^p(B_2)} = r_0^{2-n/q}\|c\|_{L^p(B')}$. Hence by {\eqref{r0small},}
$$\|\tilde b\|_{L^q(B_2)}\le 1,\quad \|\tilde c\|_{L^p(B_2)}\le 1.$$
It follows from Theorem~A
that
$$\left(\int_{B_{1}} v^\epsilon\,dy\right)^{1/\epsilon}\le C_0\left( \inf_{B_1} v + \|\tilde g\|_{L^{p}(B_2)}\right),$$ 
which gives the claim of Step 1, by scaling back to $u$ and $g$.
\smallskip

{\bf Step 2.} {\it Proof of the weak Harnack inequality \eqref{weakHarnack}.}
Set $r_0:=(3A)^{-1}$, where $A=A_R$ is defined by \eqref{defAR}.
Set $\tilde G_R:=G_R+B_{r_0}$ and denote by $X_1,\dots,X_m$ the points of the
grid $\bigl(\textstyle\frac{r_0}{2\sqrt{n}}\Z\bigr)^n \cap \tilde G_R$,
whose cardinal satisfies $$m\le C_1(n)\bigl(\textstyle\frac{R}{r_0}\bigr)^n=C_2(n)(AR)^n.$$
Set $\mathcal{B}_i:= B_{r_0}(X_i)$. Observe that the $\mathcal{B}_i$ cover $G_R$ and that $B_{2r_0}(X_i)\subset G_R'$.
It is easy to see that for any $k,\ell\in \{1,\dots,m\}$, we can connect $X_{k}$ and $X_{\ell}$  with overlapping balls as follows:
there exist an integer $$d\le C_3(n)\textstyle\frac{R}{r_0}=C_4(n)AR$$
and indices $\ell_1,\dots,\ell_d\in \{1,\dots,m\}$
such that $\ell_1=k$, $\ell_d=\ell$ and
$$\bigl|\mathcal{B}_{\ell_{i+1}}\cap\mathcal{B}_{\ell_i}\bigr|\ge C_5(n)r_0^n,\quad i=1,\dots,d-1.$$

Since
$$0<r_0=\frac13\Bigl[1+ \|b\|^{\beta_q}_{L^q_{ul}(G_R^\prime)}+\|c\|^{\gamma_p}_{L^p_{ul}(G_R^\prime)}\Bigl]^{-1}
\le \Bigl[2+\|b\|_{L^q(B')}^{\beta_q}+\|c\|_{L^p(B')}^{\gamma_p}\Bigr]^{-1}$$
with $B'=B_{2r_0}(X_{\ell_2})$, we deduce from Step 1 that
$$\left( \int_{\mathcal{B}_{\ell_2}} u^\epsilon\,dx \right)^{1/\epsilon}\le C_0 r_0^{n/\epsilon} \left( \inf_{\mathcal{B}_{\ell_2}} u
+  \|g\|_{L^{p}(B')}\right).$$
On the other hand,
$$
\left( \int_{\mathcal{B}_{\ell_1}} u^\epsilon\,dx \right)^{1/\epsilon}\ge
\left( \int_{\mathcal{B}_{\ell_1}\cap\mathcal{B}_{\ell_2}} u^\epsilon\,dx \right)^{1/\epsilon} \ge (\inf_{\mathcal{B}_{\ell_2}} u ) |\mathcal{B}_{\ell_1}\cap\mathcal{B}_{\ell_2}|^{1/\epsilon}\ge ( C_5(n)r_0^n)^{1/\epsilon} (\inf_{\mathcal{B}_{\ell_2}} u )
$$
hence, by combining the last two inequalities,
$$\left( \int_{\mathcal{B}_{\ell_2}} u^\epsilon\,dx \right)^{1/\epsilon}
\le C_0  \left(\left( \int_{\mathcal{B}_{\ell_1}} u^\epsilon\,dx \right)^{1/\epsilon} +  \|g\|_{L^{p}(B')}\right)
{\le C_0  \left(\left( \int_{\mathcal{B}_{\ell_1}} u^\epsilon\,dx \right)^{1/\epsilon} +  \|g\|_{L^{p}_{ul}(G_R')}\right).}$$
Repeating the process, we obtain
$$\left( \int_{\mathcal{B}_{\ell}} u^\epsilon\,dx \right)^{1/\epsilon}\le C_0^{d-1}
\left(\left( \int_{\mathcal{B}_{k}} u^\epsilon\,dx \right)^{1/\epsilon} +  \|g\|_{L^{p}_{ul}(G_R')}\right).$$
Since the $\mathcal{B}_\ell$ cover $G_R$,
by summing over $\ell\in  \{1,\dots,m\}$ we obtain
$$\left( \int_{G_R} u^\epsilon\,dx \right)^{1/\epsilon}\le
m^{1/\epsilon} C_0^d \left(\left( \int_{\mathcal{B}_{k}} u^\epsilon\,dx \right)^{1/\epsilon} + \|g\|_{L^{p}_{ul}(G_R')}\right).$$
 Recalling $m\le C_2(n)(AR)^n$
 and $d\le C_4(n)AR$, we have $m^{1/\epsilon} C_0^d {\ \le\ } e^{C_0 AR}$ (by readjusting $C_0$ as usual). By using Step 1 again we finally get
$$\left( \int_{G_R} u^\epsilon\,dx \right)^{1/\epsilon}\le  e^{C_0AR}\left( \inf_{\mathcal{B}_{k}} u +  \|g\|_{L^{p}_{ul}(G_R')}\right),\quad k\in  \{1,\dots,m\},$$
which implies \eqref{weakHarnack}
 since the $\mathcal{B}_k$ cover $G_R$.

\smallskip

{\bf Step 3.} {\it Proof of the local maximum principle \eqref{weakMP}.} We take $r_0:=(2A)^{-1}$, where $A=A_R$ is defined by \eqref{defAR},
and choose $x_0\in G_R$ such that
$$\sup_{G_R} u  \le \sup_{B_{r_0}(x_0)} u.$$
By using the same rescaling as in Step 1, combined with the second part of Theorem~A,
we obtain
$$ \sup_{B_1} v \le C_0\left( \left(\int_{B_{3/2}} |v|^\varepsilon\,dy\right)^{1/\varepsilon}+ \|\tilde g\|_{L^{p}(B_2)}\right).$$
By scaling back to $u$ and $g$, we get
$$\sup_{B_{r_0}(x_0)} u
\le C_0  \left(r_0^{-n/\varepsilon} \left(\int_{B_{3r_0/2}(x_0)} |u|^\varepsilon\,dx\right)^{1/\epsilon}  +  r_0^{2-n/p}\|g\|_{L^{p}(B_{2r_0}(x_0))}\right),$$
and \eqref{weakMPul} -- hence in particular \eqref{weakMP} -- follows.
\smallskip

 {\bf Step 4.} {\it  Proof of  \eqref{sharpHarnack}.}  The Harnack inequality \eqref{sharpHarnack} is a combination of the weak Harnack inequality
\eqref{weakHarnack} and the local maximum principle \eqref{weakMP}. Observe that, after rescaling or through a trivial modification of the above steps, we can replace $G_R$ in \eqref{weakHarnack} by $\tilde G_R= B_{R+1/2}$  if $G_R=B_R$ (resp. $\tilde G_R= B_{R+1/2}\setminus B_{3/2}$ if $G_R=B_R\setminus B_2$). Similarly, we can replace $G_R^\prime$ in \eqref{weakMP} by $\tilde G_R$.
\hfill $\Box$

\section{Proof of Theorem \ref{vazquez}}

We start with an elementary technical lemma, which restates the hypothesis of Theorem~\ref{vazquez} in a more convenient form for the proof of that theorem.

\begin{lem} Assume $f:[0,L]\to \R$ is  continuous for some $L>0$, and $f(0)=0$. Then
 \be\label{condgrow1}
\limsup_{s\to 0} \frac{f(s)}{s\,(\ln s)^2} <\infty
\ee
if and only if there exists $k>0$ such that
\be\label{condvazq}
e^{\sqrt{M_\delta}} = o\left(\frac{1}{\delta^k}\right)\, \quad \mbox{as $\delta\to 0$, \quad  where }
M_\delta :=\max_{s\in[0,L]} \frac{f(s)}{s+\delta}.
\ee
\end{lem}

\begin{proof}
Let us check that \eqref{condgrow1} implies \eqref{condvazq} for some $k>0$.
The assumption \eqref{condgrow1} guarantees that
$f(u)\le Cu(\ln u)^2$ on $[0,L]$ for some constant $C>0$.
Assume $\delta<e^{-2}$. Then $$[s(\ln s)^2]'=(\ln s)^2+2\ln s>0\quad\mbox{  on }\;(0,\delta],$$
hence
$$\frac{s(\ln s)^2}{s+\delta}\le \frac{\delta(\ln \delta)^2}{s+\delta}\le (\ln \delta)^2,\quad s\in(0,\delta],$$
whereas
$$\frac{s(\ln s)^2}{s+\delta}\le \max\bigl\{(\ln \delta)^2,(\ln L)^2\bigr\},\quad s\in [\delta,L].$$
Consequently,
$$\sqrt{M_\delta}\le\max_{s\in[0,L]} \sqrt{\frac{Cs(\ln s)^2}{s+\delta}}\le \sqrt{C}(|\ln \delta|+|\ln L|).$$
Choosing any $k>\sqrt{C}$, we conclude that for sufficiently small $\delta>0$ and some $\bar C>0$
$$\delta^ke^{\sqrt{M_\delta}}\le  \bar C \delta^{k-\sqrt{C}}\to 0,\ \hbox{ as } \delta\to 0.$$

Conversely, if \eqref{condvazq} holds, by setting $s=\delta$ in the definition of $M_\delta$ we get
$$
\frac{f(\delta)}{2\delta} \le M_\delta \le C_1 k^2(\log\delta)^2 \quad\mbox{ for all }\;  \delta<L,
$$
which implies \eqref{condgrow1}.
\end{proof}

 For the proof of Theorem \ref{vazquez} we need the following slight extension of Theorem~\ref{sharpharn}.

\begin{prop}\label{propext}
 Let $\Omega=B_2$. Assume \eqref{hypA},
$b, b_1, b_2, h \in L^q(B_2)$, $g\in L^p(B_2)$,
$c=c_1 + c_2$, $c_i\in L^{p_i}(B_2)$, with $n<q\le\infty$ and $p_0<p, p_i\le\infty$, $i=1,2$, and \eqref{defp0}. Set
$$A= 2+\|b\|_{L^q(B_2)}^{\beta_q}+\|c_1\|_{L^{p_1}(B_2)}^{\gamma_{p_1}}+
\|c_2\|_{L^{p_2}(B_2)}^{\gamma_{p_2}}.$$
If $u\ge0$ satisfies $\ldu\le g + \mathrm{div}(h)$, resp. $F[u]\le g$ in $B_2$, then
$$\left(\int_{B_1} u^\epsilon\,dx\right)^{1/\epsilon} \le e^{C_0A}\left( \inf_{B_1} u +  \|g\|_{L^{p}(B_{2})}+\|h\|_{L^{q}(B_{2})}\right).$$
\end{prop}

The proof of this proposition is essentially the same as Steps 1 and 2 of the proof of Theorem \ref{sharpharn}
(with $R=1$), noting that if $\tilde c(x) = r^2 c(rx)$ and $\bar p=\min\{p_1,p_2\}$ then
$$
\|\tilde c\|_{L^{\bar p}(B_2)}\le C(n)\left(r^{2-n/p_1}\|c_1\|_{L^{p_1}(B_{2r})} +r^{2-n/p_2}\|c_2\|_{L^{p_2}(B_{2r})}\right).
$$

We can now give the:

 \begin{proof}[Proof of Theorem \ref{vazquez}]
  Assume for contradiction that $u\ge 0$ is a
nontrivial supersolution
such that $\essinf_{B}u=0$ for some ball $B\subset\subset\Omega$.

\smallskip
{\bf Step 1.} In this step we will observe that, up to replacing $\Omega$ by a suitable subdomain $\Omega'$,
we can assume that $u$ is continuous.
If $u$ is a viscosity supersolution, this is so by definition.
Thus consider the case when $u$ is a weak Sobolev supersolution.

We first claim that there exists a ball
$B^\prime\subset\subset\Omega$ such that $\essinf_{B^\prime}u=0$ but the trace of $u$ on $\partial B^\prime$ does not vanish identically.
 Assume the contrary and let
 $$
 E=\Bigl\{a\in\Omega\::\: \essinf_{B_\eps(a)}u=0\ \mbox{ for all } \eps\in (0,\rho(a))\Bigr\},\ \mbox{ where }
  \rho(a)={\rm dist}(a,\partial\Omega).$$
 First note that $E$ is nonempty. Indeed, if $E$ were empty then, for each $a\in \Omega$, there would exist $\sigma(a)\in (0,\rho(a))$ such that
$\essinf_{B_{\sigma(a)}(a)}u>0$.
But since the compact $\overline B\subset\Omega$ can be covered by a finite number of balls $B_{\sigma(a_i)}(a_i)$, this would contradict $\essinf_{B}u=0$.
Next, it is clear that for each $a\in E$ we have $u=0$ a.e.~in $ B_{\rho(a)}(a)$
(since otherwise there would exist $\eta\in(0,\rho(a))$ such that the trace of $u$ on $\partial B_\eta(a)$ does not vanish identically).
It follows that the set $E$ is open.
Let $(a_i)$ be a sequence of $E$ with $a_i\to a\in \Omega$ and set $r=\rho(a)$.
We have $B(a_i,r/2)\subset\subset\Omega$ for $i$ large and $a_i\in E$, hence $u=0$ a.e.~in $B_{r/2}(a_i)$, by what we just proved.
Taking $i$ large enough we deduce that  $u=0$ a.e.~in $B_{r/4}(a)$, hence $a\in E$ and $E$ is closed in $\Omega$.
Consequently, $E=\Omega$ and $u=0$ a.e. in $\Omega$. This contradiction proves the claim.

Since $0$ is a (sub)solution we can find a {\it solution} of $\ld[\tilde u]=f(\tilde u)$ in $B^\prime$, such that $0\le \tilde u\le u$  in $B^\prime$ and $\tilde u = u$ on $\partial B^\prime$. This follows from the  general existence theory, see for instance \cite[Theorem 4.9]{Du}.
Note in that theorem it is assumed that $b_1,b_2\in L^\infty$ but what is used is that the map $u\to \mathrm{div}(b_1u) + b_2(x)\cdot Du$ is continuous from $H^1_0$ to $H^{-1}$, which is true for $b_1,b_2\in L^q$, $q>n$ (and even for $b_1,b_2\in L^n$, $n\ge3$),
owing to the Sobolev embedding and $\frac{1}{n} + \frac{n-2}{2n} + \frac{1}{2} =1$. Also, $0\le \inf_{B^\prime} \tilde u \le \mathrm{ess\,inf}_{B^\prime}u=0$ and the trace of $\tilde u$ on $\partial B^\prime$ does not vanish. Thus,  if we can prove the theorem for continuous solutions, we could apply it
with $\Omega$ replaced by $B^\prime$ and $u$ replaced by $\tilde u$ which is (H\"older) continuous by the De Giorgi-Moser theory, a contradiction.

\smallskip
{\bf Step 2.}
Set $K=\{x\in \Omega;\, u(x)=0\}$. Since $u$ is continuous, the set $K$ is closed in $\Omega$.  It is nonempty by our assumption on the existence of $B$.
Pick any $x_0\in K$ and assume $x_0=0$ without loss of generality. We are going to show that $u$ vanishes in $B_{r_1}\subset \Omega$ for some ${r_1}>0$, from which we deduce that $K$ is open in $\Omega$, so $K=\Omega$, and we are done.

We extend $f(s)=0$ for $s<0$, and set $\ue = u+ \delta$, $\fe (s) = f(s-\delta)$, $\delta\in(0,1)$. Fix $r_0>0$ such that $B_{r_0}\subset\Omega$. Then $\ue>0$ solves either
\be\label{eqep}
\mathrm{div}({ A(x)}D\ue + b_1(x)\ue)+ b_2(x) \cdot D\ue + \left(c(x) -\frac{\fe(\ue)}{\ue}\right) \ue \le {\delta} ( c(x) + \mathrm{div}(b_1))  \; \mbox{ in } B_{r_0}
\ee
or, respectively,
\be\label{eqep1}
\mm(D^2\ue)- b(x)|D\ue| {\ +}  \left(c(x) -\frac{\fe(\ue)}{\ue}\right) \ue \le {\delta}  c(x)
 \; \mbox{ in } B_{r_0}.
\ee

Hence for each $r\in (0,r_0/2]$ the rescaled function $v_r(x) = \ue(rx)$ is such that
\be\label{eqeptil}
\mathrm{div}(\tilde{A} Dv_r + \tilde{b}_1v_r)+ \tilde{b}_2(x) \cdot D v_r + \left(\tilde{c}(x) -r^2\frac{\fe(\ue)}{\ue}\right) v_r\le {\delta} ( \tilde{c}(x) + \mathrm{div}(\tilde{b}_1))  \; \mbox{ in } B_{2}{,}
\ee
and similarly for \eqref{eqep1}, where $\|\tilde{b}\|_{L^{q}(B_{2})}\le N$, $\|\tilde{c}\|_{L^{p}(B_{2})}\le N$, for some constant $N$ independent of $r$ and $\delta$ (see for instance Step 1 in the proof of Theorem \ref{sharpharn}, we can take $N$ to be the largest of $\|{b}\|_{L^{q}(B_{r_0})}$, $\|{c}\|_{L^{p}(B_{r_0})}$).

We now apply Proposition \ref{propext} to \eqref{eqeptil}, with $p_1=p$, $p_2=\infty$, $c_1=\tilde{c}$, $c_2=r^2\frac{\fe(\ue)}{\ue}$. This yields (recall $v_r(0)=u_\delta(0)=\delta$, $M_\delta$ is defined in \eqref{condvazq})
$$
\left( \int_{B_1} v_r^\epsilon\,dx \right)^{1/\epsilon}\le  \exp\left[ \bar C
 \left(1+r\sup_{[\delta, L+\delta]}\sqrt{\frac{\fe(s)}{s}} \right) \right] \left( \ue(0) + \bar C \delta \right)
= \bar C \delta \exp\left[ \bar C
 r\sqrt{M_\delta} \right],
$$
 with $\bar C= \bar C(n,\lambda,\Lambda,N)$. We {next} set $r=r_1:= \min\{r_0/2, (k\bar C)^{-1}\}$, where $k$ is the number from \eqref{condvazq}. Thus
 $$
\left( \int_{B_{r_1}} u^\epsilon\,dx \right)^{1/\epsilon}\le \left( \int_{B_{r_1}} u_\delta^\epsilon\,dx \right)^{1/\epsilon}= r_1^{n/\epsilon}\left( \int_{B_{1}} v_{r_1}^\epsilon\,dx \right)^{1/\epsilon} \le C[\delta^k e^{\sqrt{M_\delta}}]^{1/k}{.}
$$
Letting $\delta\to0$ and using \eqref{condvazq} we deduce $u\equiv0$ in $B_{r_1}$, which is what we wanted to prove.
 \end{proof}

\section{Proof of Theorem \ref{landis}}\label{secprooflandis}

In this section we consider either the divergence form operator
\begin{equation}\label{defdivagain}
\ldu:=\mathrm{div}(A(x)Du +  b_1(x)u) + b_2(x)\cdot Du +c(x) u,
\end{equation}
 where $A$ satisfies \eqref{hypA},
$b_1,b_2\in L^q_{ul}(\Omega)$, $c\in L^p_{ul}(\Omega)$, $q>n$, $p>p_0$ with \eqref{defp0}; or the fully nonlinear operator
\begin{equation}\label{defoff}
F[u]:=F(D^2u,Du,x) + c(x) u
\end{equation}
where $F(M,0,x)$ is continuous in $(M,x)$, $F(M,e,x)$ is convex (or concave) in $(M,e)$,  $F(tM,te,x) = tF(M,e,x)$ for each $t>0$, and
\begin{equation}\label{ellipoff}
\mm(M_1-M_2)-b(x)|e_1-e_2|\le F(M_1,e_1,x)-F(M_2,e_2,x)\le \mp(M_1-M_2) + b(x)|e_1-e_2|,
\end{equation}
$b\in L^q_{ul}(\Omega)$, $q>n$, $c\in L^{p}_{ul}(\Omega)$, $p>p_0$ with \eqref{defp0}. The linear and Pucci operators in \eqref{defnondiv} and \eqref{deffullynonl} are particular cases of such $F[u]$. We observe  (see \cite[Lemma 1.1]{QS}), that since $F$ is positively $1$-homogeneous and convex (resp. concave) in $(M,e)$,
\begin{eqnarray}
F(M_1,e_1,x)-F(M_2,e_2,x)&\le& F(M_1-M_2,e_1-e_2,x)\label{conv1}\\ (\hbox{resp. } F(M_1,e_1,x)+F(M_2,e_2,x)&\le& F(M_1+M_2,e_1+e_2,x)).\label{conv2}
\end{eqnarray}

We assume that $\Omega=\rn$ or $\Omega$ is an exterior domain such that (without loss)
 $B_1\subset \rn\setminus \Omega\subset B_2$.
We start by observing that the hypothesis of Theorem \ref{landis} implies the existence of a positive solution in $\Omega$.

\begin{prop}\label{prop1} Under the above hypotheses, if $\ld$ satisfies the maximum principle in each bounded subdomain of $\Omega$ then there exists
$\psi\in H^1_{\mathrm{loc}}(\overline{\Omega})$
such that $\psi>0$ and $\ld[\psi]=0$ in $\Omega$.
\end{prop}

\begin{prop}\label{prop2} Under the above hypotheses, if $F[u]$ satisfies the maximum principle in each bounded subdomain of $\Omega$ then there exists  $\psi\in W^{2,p}_{\mathrm{loc}}(\Omega)$
such that $\psi>0$ and $F[\psi]=0$ in $\Omega$.
\end{prop}

\noindent{\it Proof of Proposition \ref{prop1}}.  Fix a bounded domain  $G\subset\Omega$.
 Under our assumptions on the coefficients, it is standard that the bilinear form associated with $\ld^{(\sigma)}=\ld - \sigma$ is continuous and coercive on $H^1_0(G)$ for  $\sigma$ large enough (see \cite{GT}, \cite{T2}).  By Lax-Milgram theorem $\ld^{(\sigma)}$ is a bijection from $H^1_0(G)$ to $H^{-1}(G)$. The equation $\ldu= g+\mathrm{div}(h)$ can be written $\left(I+\sigma (\ld^{(\sigma)})^{-1}\right)u=(\ld^{(\sigma)})^{-1}(g+\mathrm{div}(h)$, and the Fredholm alternative (observe the inclusion of $H^1_0(G)$ in $H^{-1}(G)$ is compact) gives, for each $g\in L^p(G)$, $h\in L^q(G)$, a unique solution of $\ldu= g+\mathrm{div}(h)$ in $H^1_0(G)$, since the maximum principle guarantees that the kernel of $\ld$ is trivial.

In the case when $\Omega$ is an exterior domain as above,
we fix a smooth function $\phi$ such that $\phi=0$ for $|x|\le 2$, $\phi=1$ for $|x|\ge  3$ and $0\le\phi\le 1$.
If $\Omega=\rn$, we just set $\phi=1$.
For given $j\ge 2$, let $G_j=\Omega\cap B_j$, let $v_j\in H^1_0(G_j)$ be the solution of
$$\ld[v_j] = -\ld[\phi]=-\mathrm{div}(A(x)D\phi +  b_1(x)\phi) - (b_2(x){\cdot} D\phi +c(x)\phi),$$ given by the first paragraph
and set $u_j=v_j+\phi$. Then
\begin{equation}\label{BCuj}
\hbox{$\ld[u_j]=0$ in $G_j$, \ \ $u_j=0$ on $\partial\Omega$ (if nonempty), \ \ and $u_j=1$ on $\partial B_j$,}
\end{equation}
so $u_j>0$ in $G_j$ by the maximum and the strong maximum principles (the latter is a consequence of the weak Harnack inequality).
Note that, by global H\"older regularity (see \cite[Theorem 8.29]{GT}), we have $u_j\in C(\Omega\cap \overline B_j)$.
Fix an open ball $\omega\subset B_3\setminus B_2$.
We renormalize $u_j$ by considering $\tilde u_j= u_j/\|u_j\|_{L^2(\omega)}$ ($j\ge 3$), which satisfies $\ld[\tilde u_j]=0$ in $G_j$
along with $\|\tilde u_j\|_{L^2(\omega)}=1$.

Take any integer $m\ge 3$. Since $\inf_\omega \tilde u_j\le |\omega|^{-1/2}\|\tilde u_j\|_{L^2(\omega)}=c(n)$,
when $\Omega=\rn$, the Harnack inequality guarantees that
\begin{equation}\label{controljmRn}
\hbox{$0\le\tilde u_j\le  C_1(m)$ in $G_m$,\quad $j\ge m+2$.}
\end{equation}
When $\Omega$ is an exterior domain, we deduce from the Harnack inequality that
$0\le \tilde u_j\le  C_1(m)$ in $B_{m+1}\setminus B_2$ for all $j\ge m+2$,
and it then follows from \eqref{BCuj} (with $j$ replaced by $m$) and the maximum principle that
\begin{equation}\label{controljm}
\hbox{$0\le\tilde u_j\le  C_1(m)u_m$ in $G_m$,\quad $j\ge m+2$.}
\end{equation}
Now, for $j\ge m+2$, testing the equation $\ld[\tilde u_j]=0$ in $G_j$ with $\tilde u_j\theta_m^2$, where $\theta_m\in C^\infty_0(\rn)$ is such that $\theta_m=1$ for $|x|\le m/2$ and $\theta_m=0$ for $|x|\ge m$,
we get
$$\int_{G_j} \theta_m^2 D \tilde u_j\cdot AD \tilde u_j+2 \theta_m\tilde u_j D \theta_m\cdot A D \tilde u_j
=\int_{G_j}  \theta_m^2\tilde u_j(b_2-b_1) \cdot D \tilde u_j-\tilde u_j^2 b_1\cdot D(\theta_m^2)+c\tilde u_j^2\theta_m^2.$$
Using Young's inequality, \eqref{hypA} and then H\"older's inequality, we easily deduce that
$$
\begin{aligned}
\frac{\lambda}{2}\int_{G_{m/2}} |D \tilde u_j|^2
&\le  C_2(m)\int_{G_m} (1+|b_2-b_1|^2+|b_1|+|c|)\tilde u_j^2 \\
&\le  C_2(m)\Bigl(1+\bigl\||b_2-b_1|^2+|b_1|+|c|\bigr\|_{L^{n/2}(G_m)}\Bigr)\|\tilde u_j\|^2_{L^{2^*}(G_m)}.
\end{aligned}
$$
From our assumptions on the coefficients and \eqref{controljm} (resp., \eqref{controljmRn}), it follows that for all $j \ge m+2$,
$$\|\tilde u_j\|_{H^1(G_{m/2})}\le C(m)(1+\lambda^{-1/2})\|u_m\|_{L^{2^*}(G_m)}<\infty\quad \hbox{(resp., $\le C(m)(1+\lambda^{-1/2})$)}$$
(recall that $u_m=v_m+\phi\in L^{2^*}(G_m)$ due to $v_m\in H^1_0(G_m)$ and Sobolev's imbedding).
Therefore, for each $m\ge 3$, the sequence $\{\tilde u_j\theta_m\}_{j=1}^\infty$ is bounded in $H^1_0(\Omega)$.
By a diagonal procedure, we deduce that $\tilde u_j$ has a subsequence which converges weakly in $H^1_{loc}(\overline\Omega)$ and strongly in $L^2_{loc}(\overline\Omega)$
to a nonnegative solution $\psi$ in the whole $\Omega$, and $\psi$ is nontrivial due to $\|\psi\|_{L^2(\omega)}=1$. As we recalled in Remark 1.1, $\psi$ is H\"older continuous in $\Omega$.
Finally,  we have $\psi>0$ by the SMP.  \hfill $\Box$ 

\begin{rem}
 Although we shall not use this fact, we note that, in the case when $\Omega$ is an exterior domain,
the function $\psi$ obtained in Proposition \ref{prop1} actually satisfies homogeneous boundary conditions, in the sense that
$\varphi\psi\in H^1_0(\Omega)$ for all $\varphi\in C^\infty_0(\rn)$.
\end{rem}

\medskip

\noindent{\it Proof of Proposition \ref{prop2}}. Fix a bounded smooth domain  $G\subset\Omega$. Since $\|(c-\sigma)^+\|_{L^p(G)}\to 0$ as $\sigma\to\infty$, we can fix $\sigma$ large enough so that the operator $F[u]-\sigma u$ satisfies the hypotheses of \cite[Theorem 1]{S1}, in particular, condition (6) there. By that theorem and well-known regularity results (see the Remark below) for each $v\in C(\overline{G})$, $g\in L^p(G)$, there is a unique $u\in W^{2,p}_{\mathrm{loc}}(G)\cap C^\alpha(\overline{G})$ such that $F[u]-\sigma u = g-\sigma v$  in $G$ and $u=0$ on $\partial G$. The operator $S:C(\overline{G})\to C(\overline{G})$ given by $S[v]=u$ is thus well defined and compact.

We briefly recall the Leray-Schauder alternative.
\begin{thm}[Corollary 1.19, \cite{CQ}]
	Let $S:X\to X$ be compact, where X is a Banach space. Then one of the following holds:

	(i) $x - tS(x) = 0$ has a solution for  every $t\in [0,1]$, or

	(ii) the set $\{x : \exists\ t \in[0,1] : x-tS(x)=0\}$ is unbounded.
\end{thm}

If the alternative (ii) happened for our $S$, we would have a sequence $u_n$ such that $\|u_n\|_{C(\overline{G})}\to \infty$ and $F[u_n]-\sigma u_n = t_n(g-\sigma u_n)$ in $G$, $u_n=0$ on $\partial G$, for some $t_n\in [0,1]$. Setting $\tilde u_n = u_n / \|u_n\|_{C(\overline{G})}$, we see that $\|u_n\|_{C^\alpha(\overline{G})} \le C$, by global H\"older regualrity (see \cite[Theorem 6.2]{KSh}). Passing to the limit  along
 a subsequence, using the stability properties of viscosity solutions with respect to uniform convergence (\cite[Proposition 9.4]{KSh}), we find a solution of $F[\tilde u] -(1-t_0)\sigma \tilde u = 0$ in $G$, $\tilde u=0$ on $\partial G$, and $\tilde u\not \equiv 0$, which contradicts the maximum principle.  Note $(1-t_0)\sigma\ge 0$ so $F -(1-t_0)\sigma $ satisfies the maximum principle if $F$ does.

So by the Leray-Schauder alternative the equation $F[u] = g$ in $G$, $ u=0$ on $\partial G$, has a solution for each $g\in L^p(G)$, which is strong by regularity ($u\in W^{2,p}_{\mathrm{loc}}(G)\cap C^\alpha(\overline{G})$), and then unique by the maximum principle.  We now  solve $F[v] = -c$ in $G$, $ v=0$ on $\partial G$, and set $u=v+1$. Then $F[u]=0$ in $G$ and $u=1$ on $\partial G$ so $u>0$ in $\overline{G}$, by the maximum and the strong maximum principle.

Take an increasing sequence of smooth domains $\Omega_j\subset\Omega$ which converges to $\Omega$. Set $G_j=\Omega_j\cap B_j$ and take the  solutions $u_j$ given by the above procedure in $G_j$. Fix a point $x_0\in G_1$ and replace $u_j$ by $\tilde u_j = u_j/u_j(x_0)$, $F[\tilde u_j]=0$ in $G_j$, $\tilde u_j>0$ in $G_j$, $\tilde u_j (x_0)=1$. By the Harnack inequality, for each compact $K\subset\Omega$, 
 we have $0<u_j\le C(K)$ in $K$, for $j>j_0(K)$, where $j_0$ is such that $K\subset \Omega_{j_0}$. By regularity $u_j$ is bounded in $W^{2,p}(K)$, and hence (up to a subsequence) converges weakly in that space and strongly in $C(K)$ (recall $p>n/2$).  Recalling the stability of viscosity solutions with respect to uniform convergence, and that the operator $u\to b|Du| + cu$ is precompact from $W^{2,p}$ to $L^p$ for $b\in L^q$, $c\in L^{p}$, $q>n$, $p>n/2$, by the embeddings $W^{2,p}\hookrightarrow W^{1,np/(n-p)} \hookrightarrow C^\alpha$, we deduce by a diagonal procedure that a subsequence of $u_j$ converges weakly in $W^{2,p}_{\mathrm{loc}}(\Omega)$  and strongly in $C_{\mathrm{loc}}(\Omega)$ to a strong solution $\psi$ in $\Omega$,
with $\psi(x_0)=1$.
Finally, we have $\psi>0$ by the SMP.
 \hfill $\Box$

\begin{rem} For the reader's convenience we quote precisely the results from
the general theory of fully nonlinear equations with measurable coefficients. The bases of the theory for bounded ingredients can be found in \cite{CCKS}, \cite{CKS}, \cite{W}. Extensions to  unbounded coefficients, in the generality which we require were obtained as follows: the ABP inequality can be found in \cite{KSABP}, the Harnack inequality in \cite{KSh}, \cite{KSl},  the global $C^\gamma$ estimates in \cite[Theorem 6.2]{KSh}, the strong solvability and $W^{2,p}_{\mathrm{loc}}$-estimates for extremal equations in \cite[Theorem 7.1]{KSh} (the same proof applies to  convex/concave operators) the stability of viscosity solutions with respect to uniform convergence in \cite[Proposition 9.4]{KSh}, global $W^{2,p}$ estimates follow either from \cite{W} (with a straightforward extension to unbounded coefficients using the already available global $C^\gamma$ estimates and stability) or from  \cite[Proposition 2.4]{KSgl}. Note also that the results in \cite{S1} were stated for $c,f\in L^n$, however all ingredients of the proofs in \cite{S1}, which we just quoted, were later established for $c,f\in L^p$, $p>p_0$. \end{rem}
\medskip

\noindent{\it Proof of Theorem \ref{landis}}. If the elliptic operator is in divergence form, let $\psi$ be the function given by Proposition \ref{prop1} (recall $\psi$ is continuous, see Remark \ref{remTh2}). In the fully nonlinear case, if $F$ is convex in the derivatives of $u$, let $\psi$ be the function given by Proposition \ref{prop2} applied to $F[u]$; whereas if $F$ is concave let $\psi$ be the function given by Proposition \ref{prop2} applied to $\tilde F[u]=-F[-u]$. In all cases,  normalize $\psi$ so that $\psi(x_0)=1$ for some point $x_0\in \rn\setminus B_2\subset \Omega$. Then the sharp Harnack inequality, \eqref{sharpHarnack} in Theorem~{\ref{sharpharn}} with $g=0$, gives  ($G_R$ is defined in Theorem~\ref{sharpharn})
$$
\inf_{G_R} \psi \ge e^{-C_1R}.
$$
Let $u$ be a solution as in Theorem \ref{landis}, with $u\le 0$ on $\partial \Omega$ if the latter is not empty (or replace $u$ by $-u$, and $F$ by $\tilde F$ in the fully nonlinear case). Fix $\delta>0$. Under assumption \eqref{concllandis2} (in particular if \eqref{concllandis} is true),
there exists a sequence $R_i>3$, $R_i\to\infty$ 
such that
\begin{equation}\label{udeltapsi}
\hbox{ $u<\delta\psi $ on $\partial B_{R_i}$.}
\end{equation}

With $\delta$ and $i$ being fixed, our goal is now to apply the maximum principle
to show that $u\le\delta\psi$ in $\Omega_i:=\Omega \cap B_{R_i}$.
We shall not work directly with $u-\delta\psi$,
because difficulties arise in order to verify the boundary conditions on $\partial\Omega$ when $\Omega$ is nonsmooth.
It turns out that this can be circumvented by considering instead $z:=u_+-\delta\psi$ (in the divergence case) or $\tilde z:=(u-\delta\psi)_+$
(in the non-divergence case).

First of all, we observe that
\begin{equation}\label{Lzge0}
\mathcal{L}_D[z]\ge 0 \quad\hbox{ (resp., $F[\tilde z]\ge 0$)\quad in $\Omega$.}
\end{equation}
Indeed, in the divergence case,
this follows from
the fact that $D(u_+)=\chi_{\{u>0\}}\, Du$, along with
Kato's inequality
$$\mathrm{div}(A(x)Du^+)\ge \chi_{\{u>0\}}\ \mathrm{div}(A(x)Du)$$
(in the weak Sobolev sense). In the nondivergence case,
\eqref{Lzge0} follows from the fact that the maximum of the two viscosity (sub-) solutions $u-\delta\psi$ and $0$ is a viscosity subsolution. To check that $u-\delta \psi$ is a subsolution, i.e. $F[u-\delta\psi]\ge0$,  we use \eqref{conv1} if $F$ is convex, resp. \eqref{conv2} if $F$ is concave, as well as the choice of the strong solution $\psi$ we made above.

We next turn to the boundary conditions
on $\partial \Omega_i=\partial \Omega \cup \partial B_{R_i}$.
First considering the divergence case, we shall check that $z \le 0$ on $\partial \Omega_i$, which means
\begin{equation}\label{zplusH10}
z_+\in H^1_0(\Omega_i).
\end{equation}
Since no smoothness is assumed on $\Omega$, we cannot use traces on $\partial\Omega$ and some care is needed.
Fix a smooth function $\varphi$ such that
$\varphi=1$ for $|x|\le 2$, $\varphi=0$ for $|x|\ge 3$ and $0\le \varphi\le 1$.
By assumption (cf.~Remark~\ref{remTh2}), we have $u^+\varphi\in H^1_0(\Omega)$,
hence actually $u^+\varphi\in H^1_0(\Omega_i)$. Therefore there exists a sequence $\theta_j\in C^\infty_0(\Omega_i)$ such that
$\theta_j\to u^+\varphi$ in $H^1(\Omega)$.
Now setting $\phi_j:=\theta_j-\delta\varphi\psi+(1-\varphi)(u_+-\delta\psi)$
and recalling $\psi\in H^1_{loc}(\overline\Omega)$ and $u\in H^1_{loc}(\Omega)$,
 we have $\phi_j\in H^1(\Omega_i)$ with
$\phi_j\to u_+-\delta\psi=z$, hence $(\phi_j)_+\to z_+$, in $H^1(\Omega_i)$. But, using \eqref{udeltapsi}, $\varphi,\psi\ge0$, and
the continuity of $u$ and $\psi$ in $\Omega$ (cf.~Remark \ref{remTh2}), in particular near $\partial B_{R_i}$,
we easily check that $(\phi_j)_+\in H^1(\Omega_i)$ vanishes in a neighborhood of $\partial \Omega$, as well as on $\partial B_{R_i}$, 
hence $(\phi_j)_+\in H^1_0(\Omega_i)$.
This guarantees \eqref{zplusH10}. In the nondivergence case, we need to show that
\begin{equation}\label{tildezle0}
\tilde z\in C(\overline\Omega_i)\quad\hbox{with $\tilde z\le 0$ on $\partial\Omega_i$.}
\end{equation}
But since $u\in C(\overline\Omega)$ with $u\le 0$ on $\partial \Omega$ and $\psi\in C(\Omega)$,
property \eqref{tildezle0} follows from \eqref{udeltapsi} and the fact that $0\le \tilde z=(u-\delta\psi)_+\le u$.

Now, in view of \eqref{Lzge0}-\eqref{tildezle0}, we deduce from the maximum principle that
$z\le 0$ or $\tilde z\le 0$, hence $u\le \delta\psi $ in $\Omega \cap B_{R_i}$. Letting $i\to\infty$ and then $\delta \to 0$ we conclude that
$u\le 0$.  
But then Theorem \ref{sharpharn} applies to $-u$, 
hence, after the normalization $u(x_0)=-1$, either $u\equiv0$ or  $
\inf_{G_R} |u| \ge e^{-C_1R}$ for all $R>2$, a contradiction with assumption \eqref{concllandis2}.

\section{Appendix. Proof of Theorem A.}

First, in the divergence case, Theorem A follows from \cite{T2} (or see Remark at the end of \cite[Section 8.10]{GT}). So we concentrate here on the non divergence (fully nonlinear) case.

We recall the ABP inequality: if diam$(\Omega)\le 1$, $\|b\|_{L^q(\Omega)}\le 1$, $c,g\in L^p(\Omega)$, $p>p_0$, $b,c,g\ge0$, and $w\in C(\overline{\Omega})$ is a viscosity solution of
$$
\mp(D^2w) + b|Dw|-cw\ge -g\quad\mbox{ in }\; \Omega, \qquad u\le 0\quad\mbox{ on }\; \partial \Omega,
$$
then for some $C_0=C_0(n,q,\lambda, \Lambda)$
$$
\sup_\Omega w\le C_0\|g\|_{L^p(\Omega^+)}, \qquad \Omega^+=\{x\in\Omega\::\: w(x)>0\}.
$$
This follows from Theorem 2.9 in \cite{KSABP}, applied on each connected component of $\Omega^+$ (noting that $-cw\le0$ on that set).

We now  prove Theorem A. We start with the weak Harnack inequality, and provide a proof based on the original approach by Krylov and Safonov, \cite{Saf}. Note that $F[u]\le g$, \eqref{ellipoff} and $u\ge0$ imply
\begin{equation}\label{extreq}
\mathcal{M}_{\lambda,\Lambda}^-(D^2u) -b|Du|-c^-u\le g, \qquad u\ge0.
\end{equation}
  We assume that \eqref{extreq} holds in  $B_2$, $\|b\|_{L^q(B_2)}\le 1$, $\|c\|_{L^{p}(B_2)}\le 1$, and want to prove \eqref{weakha}.

In the following all constants will be allowed to depend on $n,p,q,\lambda,\Lambda$.  By a simple covering argument, it is enough to show that there exists $\rho_0\in (0,1/2)$ such that for each $x_0\in B_1$
\begin{equation}\label{ineqrho0}
\left(\int_{B_{\rho_0}(x_0)} u^\epsilon\,dx\right)^{1/\epsilon}\le C\left( \inf_{B_{\rho_0}(x_0)} u + \|g\|_{L^{p}(B_{2\rho_0}(x_0))}\right).
\end{equation}

\begin{prop}\label{growth} There exist positive constants $\rho_0, \kappa, \delta, \bar C$ such that for each $x_1\in  B_1$ and each $\rho\in (0,\rho_0]$, if for some $a>0$
\begin{equation}\label{ineqrho}
|\{u> a\}\cap B_{\rho}(x_1)|\ge (1-\delta) |B_{\rho}(x_1)|
\end{equation}
then
\begin{equation}\label{conclrho}
\inf_{B_{\rho}(x_1)}u > \kappa a - \bar C \rho^{2-n/p}\|g\|_{L^{p}(B_{2\rho}(x_1))}.
\end{equation}
\end{prop}

\begin{proof} Rescaling $x\to (x-x_1)/\rho$ we can assume that $u$ is a nonnegative solution of
\begin{equation}\label{extreq1}
\mathcal{M}_{\lambda,\Lambda}^-(D^2u) -b_\rho|Du|-c_\rho u\le g_\rho, \quad \mbox{in }\;B_2
\end{equation}
 where $b_\rho(x)=\rho b(x_1+\rho x)$,  $c_\rho(x)=\rho^2 c(x_1+\rho x)$, $g_\rho(x)=\rho^2 g(x_1+\rho x)$ and
 $\|b_\rho\|_{L^q(B_2)} = \rho^{1-n/q}\|b\|_{L^q(B_{2\rho}(x_1))}\le \rho^{1-n/q}$, $\|c_\rho\|_{L^p(B_2)} \le \rho^{2-n/p}$.
The choice of $\rho_0$ will be made so that these norms be sufficiently small.

Assume first $a=1$ and set $v(x) = 1-|x|^2$. We have $|\{u> 1\}\cap B_1|\ge (1-\delta) |B_1|$, hence in particular
$$
|\{v>u\}\cap B_1|\le \delta|B_1|.$$
We have in $B_1$
\begin{eqnarray*}
\mp(D^2(v-u)) + b_\rho|D(v-u)|-c_\rho(v-u)&\ge& \mm(D^2v) - b_\rho|Dv| - c_\rho v - g_\rho\\
&\ge & -C( 1+ b_\rho + c_\rho) - g_\rho,\end{eqnarray*}
and $v-u\le 0$ on $\partial B_1$, so by the ABP inequality (applied with $p$ substituted by $n$ in case $p>n$)
$$
\sup_{B_1} (v-u) \le C(|\Omega^+|^{1/p} + \|b_\rho\|_{L^q(B_2)} + \|c_\rho\|_{L^p(B_2)}) + C_0 \|g_\rho\|_{L^p(B_2)},
$$
where $|\Omega^+| = |\{v-u>0\}|\le \delta |B_1|$. By choosing $\delta$ and $\rho_0$ sufficiently small
we obtain
$$
\sup_{B_1} (v-u) \le \frac{1}{4} + C_0 \|g_\rho\|_{L^p(B_2)}.
$$
Since $v\ge 3/4$ in $B_{1/2}$ we get
$$
u\ge \frac{1}{2} - C_0 \|g_\rho\|_{L^p(B_2)}\quad \mbox{in }\;B_{1/2}.
$$
For arbitrary $a>0$,  
by replacing $u$ by $\tilde u = u/a$ we get
\begin{equation}\label{tildeu12}
\tilde u \ge \frac{1}{2} - \frac{C_0}{a} \|g_\rho\|_{L^p(B_2)}\quad \mbox{in }\;B_{1/2}.
\end{equation}
Now, if $a<4C_0\|g_\rho\|_{L^p(B_2)}$ the inequality \eqref{conclrho} trivially holds with the choice $\kappa = 1$, $\bar C= 4C_0$ (since its right-hand side is negative). If on the other hand $a\ge 4C_0\|g_\rho\|_{L^p(B_2)}$ we have $\tilde u \ge 1/4$ in $B_{1/2}$. We then take
$$
w(x)=\frac{1}{4}\:\displaystyle\frac{|x|^{-s}-2^{-s}}{(1/2)^{-s}
-2^{-s}}\,.
$$
where $s>0$ is such that $\mm(D^2(|x|^{-s}))=0$, that is, $\lambda (s+1)=\Lambda (N-1) $. As above
\begin{eqnarray*}\mp(D^2(w-\tilde u)) + b_\rho|D(w-\tilde u)|-c_\rho(w-\tilde u)&\ge& \mm(D^2w) - b_\rho|Dw| - c_\rho w - g_\rho/a\\
&\ge & -C( b_\rho + c_\rho) - g_\rho/a,\end{eqnarray*}  
 with $w-\tilde u\le 0$ on $\partial(B_2\setminus B_{1/2})$,
so by the ABP inequality for each $\ep>0$ there exists $\rho_0>0$ such that if $\rho\in (0,\rho_0]$
$$
\sup_{B_2\setminus B_{1/2}} (w-u) \le \ep - \frac{C_1}{a} \|g_\rho\|_{L^p(B_2)}
$$
Setting $\ep =\frac{1}{2}\min_{B_1\setminus B_{1/2}} w$ we obtain
$$
\tilde u\ge \frac{\ep}{2} - \frac{C_1}{a} \|g_\rho\|_{L^p(B_2)}\quad \mbox{in }\;B_1\setminus B_{1/2},
$$
Combining this with \eqref{tildeu12}, and choosing $\kappa=\frac14 \min(1,\ep)$, $\bar C=\max\{4C_0,C_1\}$, we deduce that
$$
u> \kappa a- \bar C\|g_\rho\|_{L^p(B_2)}\quad \mbox{in }\;B_1,
$$
which concludes the proof of Proposition \ref{growth}.\end{proof}

Next we prove \eqref{ineqrho0}, for the number $\rho_0$ given by Proposition \ref{growth}. By replacing $u$ by $(\inf_{B_{\rho_0}(x_0)} u +\alpha +  \|g\|_{L^{p}(B_{2\rho_0}(x_0))})^{-1} u$ we see that it is enough to assume that $\inf_{B_{\rho_0}(x_0)} u\le 1$, $\|g\|_{L^{p}(B_{2\rho_0}(x_0))}\le 1$, and prove that
\begin{equation}\label{meee}
\left(\int_{B_{\rho_0}(x_0)} u^\epsilon\,dx\right)^{1/\epsilon}\le C
\end{equation}
with constants independent of $\alpha>0$ (then let $\alpha\to0$). Recall $\rho_0\in(0,1/2)$ depends only on $n,p,q,\lambda,\Lambda$.


It follows from Proposition \ref{growth} that we can find a constant $ M>1$ such that
\begin{equation}\label{ineqrho1}
|\{u> M\}\cap B_{\rho_0} |\le (1-\delta) |B_{\rho_0} |,\quad B_{\rho_0} =B_{\rho_0}(x_0).
\end{equation}
Indeed, if \eqref{ineqrho1} failed, by Proposition \ref{growth} we would have, setting $M= 1+(1/\kappa)(\bar C+2)$, that
$u\ge \kappa M - \bar C\ge 2$ in $B_{\rho_0}$ which is a contradiction with $\inf_{B_{\rho_0}} u\le 1$. Note we also have $\kappa M^k - \bar C\ge 2 M^{k-1}$ , for each $k\ge1$.

 We now apply a well-known argument, to prove by induction that  for all $k\in \mathrm{N}$, $k\ge1$,
\begin{equation}\label{ineqrho2}
|\{u> M^k\}\cap B_{\rho_0} |\le (1-c(n)\delta)^k |B_{\rho_0}|,
\end{equation}
for some (small) constant $c(n)>0$. Specifically, we use the Krylov-Safonov "propagating ink-spots lemma" (\cite[Lemma 1.1]{Saf}), in the form given for instance in \cite[Lemma 2.1]{IS}:
\begin{lem}
Let $E\subset F \subset B_{\rho_0}$ be open sets. Assume for some $\delta>0$ we have $|E|\leq (1-\delta) |B_{\rho_0}|$, and for any ball  $ B\subset B_{\rho_0}$, if $|B\cap E|>(1-\delta)|B|$ then $ B\subset F$. Then $|E|\leq (1-c\delta)|F|$,
for some constant $c=c(n)>0$.
\end{lem}
The induction proceeds by setting $E=\{u> M^{k}\}\cap B_{\rho_0}$, $F=\{u> M^{k-1}\}\cap B_{\rho_0}$, $k\ge1$. The condition of the lemma is guaranteed by Proposition \ref{growth} and the choice of $M$ we made.

Then by \eqref{ineqrho2} there exists $\varepsilon^\prime>0$ such that
$
|\{u\geq t\}\cap B_{\rho_0} |\leq
C\min\{1,t^{-\varepsilon^\prime}\}$, $  t>0$.
Indeed, if $t\leq 1$ set $C=|B_{\rho_0}|$. If $t>1$ let $j\in \mathrm{N}$ such that $M^{j-1} \le t<M^j$, so
$$
|\{u\geq t\}\cap B_{\rho_0} |\leq (1-c\delta)^{j-1}|B_{\rho_0}|=M^{-\varepsilon^\prime (j-1)}|B_{\rho_0}|\leq M^{\varepsilon^\prime} |B_{\rho_0}|  t^{-\varepsilon^\prime}= \frac{|B_{\rho_0}|}{1-c\delta} t^{-\varepsilon^\prime},
$$
if we choose $\varepsilon^\prime$ so that  $M^{-\varepsilon^\prime}=1-c\delta$. Now, take $\epsilon=\varepsilon^\prime/2$, then
$$
\int_{B_{\rho_0}} u^{\epsilon} = \frac{\varepsilon^\prime}{2} \int_0^{\infty} t^{\frac{\varepsilon^\prime}{2}-1} |\{ u\geq t\}\cap B_{\rho_0}| \, dt
\leq C \int_0^{\infty} t^{\frac{\varepsilon^\prime}{2}-1} \min \{ 1, t^{-\varepsilon^\prime} \} \, dt=C.
$$
\medskip

In the end we prove the local maximum principle (LMP) in Theorem A. We have $\mp(D^2u) + b|Du| \ge -cu+f $ so by the already known LMP (see \cite{KSl}), setting $p_1=(p+p_0)/2>p_0$,
\begin{eqnarray*}\sup_{B_1} u&\le&  C\left(\left(\int_{{ B_{5/4}}} |u|^\varepsilon\,dx\right)^{1/\varepsilon} + \|cu\|_{L^{p_1}( B_{4/3})}+ \|f\|_{L^{p_1}( B_{4/3})}\right)\\
&\le&   C\left(\left(\int_{{B_{5/4}}} |u|^\varepsilon\,dx\right)^{1/\varepsilon} + \|u\|_{L^{A}(B_{4/3})}+ \|f\|_{L^{p}(B_{4/3})}\right),\end{eqnarray*}
where $1/A = 1/p_1 - 1/p$ is given by the H\"older inequality (recall $\|c\|_{L^{p}(B_{2})}=1$). Finally, we downgrade the $L^A$-norm of $u$ on ${B_{4/3}}$ to a ``$L^\varepsilon$-norm" on  $B_{3/2}$ through a well-known analysis argument on an expanding sequence of balls (given for instance on pages 74-76 of \cite{HL}).

\end{document}